\documentclass[11pt]{article}
\usepackage[margin=1in]{geometry}

\usepackage[utf8]{inputenc} 
\usepackage[T1]{fontenc}
\usepackage{url}
\usepackage{ifthen}
\usepackage{appendix}
\usepackage[backend=biber,style=alphabetic]{biblatex}
\usepackage{comment}
\usepackage{algorithmic}
\usepackage{amsmath}
\usepackage{amsthm}
\usepackage{amssymb}
\usepackage{color}
\usepackage{mathtools}
\usepackage{makecell}
\usepackage{bm}
\usepackage{todonotes}
\usepackage{diagbox}
\usepackage{tablefootnote}

\DeclareMathOperator{\Var}{Var}

\DeclareMathOperator{\Ber}{Ber}
\DeclareMathOperator{\Rad}{Rad}

\let\baraccent=\= 
\renewcommand{\=}[1]{\stackrel{#1}{=}} 

\providecommand{\RR}{\mathbb{R}}

\providecommand{\ZZ}{\mathbb{Z}}

\providecommand{\cG}{\mathcal{G}}
\providecommand{\cL}{\mathcal{L}}

\providecommand{\cP}{\mathcal{P}}

\providecommand{\PP}{\mathbb{P}}
\providecommand{\EE}{\mathbb{E}}

\providecommand{\e}{\epsilon}
\providecommand{\eps}{\epsilon}

\mathchardef\mhyphen="2D 

\providecommand{\sm}{\setminus}

\newcommand{\interior}[1]{%
  {\kern0pt#1}^{\mathrm{o}}%
}
\newcommand\independent{\protect\mathpalette{\protect\independenT}{\perp}}
\def\independenT#1#2{\mathrel{\rlap{$#1#2$}\mkern2mu{#1#2}}}
\providecommand{\indy}{\independent}

\newtheorem{theorem}{Theorem}[section]
\newtheorem{lemma}[theorem]{Lemma}
\newtheorem{definition}[theorem]{Definition}

\newtheorem{conjecture}[theorem]{Conjecture}
\newtheorem{claim}[theorem]{Claim}
\newtheorem{proposition}[theorem]{Proposition}
\newtheorem{corollary}[theorem]{Corollary}

\usepackage{thmtools}
\usepackage{thm-restate}

\begin{document}
\title{Subadditivity Beyond Trees and the Chi-Squared \\ Mutual Information} 


\author{Emmanuel Abbe \\ EPFL
  \and
  Enric Boix Adser\`a \\ MIT
}


\maketitle

\begin{abstract}
In 2000, Evans et al.\ \cite{evans} proved the subadditivity of the mutual information in the broadcasting on tree model with binary vertex labels and symmetric channels. They raised the question of whether such subadditivity extends to loopy graphs in some appropriate way. We recently proposed such an extension that applies to general graphs and binary vertex labels \cite{abbe2018information}, using synchronization models and relying on percolation bounds. This extension requires however the edge channels to be symmetric on the product of the adjacent spins. A more general version of such a percolation bound that applies to asymmetric channels is also obtained in \cite{yury_percolation}, relying on the SDPI, but the subadditivity property does not follow with such generalizations.

In this note, we provide a new result showing that the subadditivity property still holds for arbitrary (asymmetric) channels acting on the product of  spins, when the graphs are restricted to  be series-parallel. The proof relies on the use of the Chi-squared mutual information rather than the classical mutual information, and various properties of the former are discussed. 

We also present a generalization of the broadcasting on tree model (the synchronization on tree) where the bound from \cite{yury_percolation} relying on the SPDI can be significantly looser than the bound resulting from the Chi-squared subadditivity property presented here.
\end{abstract}

\tableofcontents


\section{Introduction}


\subsection{Spin synchronization model}
We consider the problem of reconstructing $n$ independent, uniform $(\pm 1)$-valued spins $X_1,\dots,X_n \stackrel{i.i.d.}{\sim} \Rad(1/2)$ living on the vertices of an $n$-vertex graph $G$, by observing their interactions $Y_{e}$ on the edges of the graph. For each $e = (i,j) \in E(G)$, the interaction $Y_e = Y_{ij}$ depends only on $X_i \cdot X_j$. Formally, we can factor the joint distribution as $$P_{X,Y}(x,y) = \left(\prod_{(i,j) \in E(G)} P_{Y_{ij}|X_i \cdot X_j}(y_{ij}|x_i \cdot x_j)\right) \cdot P_{X}(x).$$ In other words, each $Y_{ij}$ is the output of $X_i \cdot X_j$ through a channel $Q_{ij} := P_{Y_{ij} | X_i \cdot X_j}$.

\begin{definition}
We refer to $(X,Y)$ as a spin synchronization model on a graph $G$ with edge channels $Q$.
\end{definition}

This spin synchronization model has previously appeared in \cite{our_grid}, for example. Note that in \cite{our_grid} the alphabets were not restricted to be binary, but were arbitrary groups instead.

Given a spin synchronization problem $(X,Y)$, the goal is to reconstruct the spin $X_u$ for some $u \in V(G)$, given $Y$ and $X_W$ where $W \subset V(G)$ (due to symmetry, we may as well freeze at least one reference vertex). It is thus natural to look at some measure of dependency between to generic vertices in the graph, and a natural quantity of interest is the mutual information $I(X_u; X_W, Y)$.

\paragraph*{Connection to other reconstruction models} 
Depending on the choices of the graph and the edge interaction channels, the spin synchronization model captures models such as (a) broadcasting on trees \cite{evans}, (b) censored block models \cite{heimlicherlelargemassoulie,abbs}, (c) synchronization on grids \cite{our_grid}, and (d) spiked Wigner models \cite{yash_sbm}.
Thus, several information-theoretic reconstruction thresholds for these models can be obtained by studying the information-theoretic thresholds for the spin synchronization model. We refer the reader to \cite{abbe2018information,yury_percolation} for details on these connections.

\paragraph*{Connection to Ising Model} The spin synchronization model is also related to the Ising model in statistical physics; conditioned on the edge observations $Y$, the posterior distribution of the vertex spins $X$ is given by an Ising model. However, we will interested here in the average-case behavior over the edge observations in the model, so naively applying correlation decay results on Ising models (e.g., Dobrushin conditions \cite{dobrushin1968problem}) would yield weaker bounds than those we will obtain.

\subsection{Main result}

The main contribution of this paper is Theorem \ref{thm:mainseriesparallelbound}, a new impossibility result for reconstruction in the spin synchronization model.\footnote{Part of the results in this note appeared in Enric Boix's undergraduate senior thesis \cite{enric2018thesis}.} This result is incomparable to the previously-known bounds derived in \cite{abbe2018information} and \cite{yury_percolation}. We discuss in detail the relationship between the new bound and the previous bounds, and in Section \ref{sec:generalizedbot} we give examples of spin synchronization problems for which the new bound is tighter. In Section \ref{sec:botapplication}, we also use Theorem \ref{thm:mainseriesparallelbound} to rederive the impossibility result of \cite{evans} for broadcasting on trees.

For our main result, we work with the Chi-squared mutual information, $I_2$, instead of the classical KL mutual information, $I_{KL}$. See Appendix~\ref{app:chi2mutualinformation} for definitions. As we will see, the Chi-squared mutual information turns out to be a more natural choice for the proof of the theorem. In any case, for our purposes the KL and Chi-squared mutual informations are equal up to a factor of 2, because of the following well-known inequalities $\frac{1}{2}I_2(A; U) \leq I_{\mathrm{KL}}(U; A) \leq I_2(U; A)$ when $U \sim \Rad(1/2)$, reproved in Appendix \ref{app:chi2mutualinformation}.

We are ready to state our main result.

\begin{theorem}\label{thm:mainseriesparallelbound} Let $(X,Y)$ be a spin synchronization model on a series-parallel graph $G$ with terminals $u$ and $v$, and arbitrary edge channels $Q$. Then $$I_2(X_u; X_v \mid Y_{E(G)}) \leq \sum_{P \in \cP_G(u,v)} I_2(X_u; X_v \mid Y_{E(P)}).$$ Here $\cP_G(u,v)$ is the set of paths (self-avoiding walks) from $u$ to $v$ in $G$.
\end{theorem}

\subsection{Comparison to previous results}
\paragraph{Comparison to \cite{abbe2018information}} The recent work \cite{abbe2018information} derives a subadditivity bound similar to Theorem \ref{thm:mainseriesparallelbound}. While the bound of \cite{abbe2018information} applies to spin synchronization models on general graphs, it requires the edge channels of the models to have certain symmetry. In contrast, in Theorem \ref{thm:mainseriesparallelbound}, we handle general channels, while giving up generality on the base graph. For ease of comparison, we state the result of \cite{abbe2018information} below.


\begin{theorem}[ \cite{abbe2018information}]\label{thm:mainsymmetricbound}
Let $(X,Y)$ be a spin synchronization model on a graph $G$, such that the edge channels $Q$ are symmetric: for each $e \in E(G)$, there is a measurable transformation $T_e = T_e^{-1}$ such that $Q_e(\cdot | +1) = T_e \sharp Q_e(\cdot | -1)$. ($T_e\sharp$ denotes the push-forward operation.)

Then for all $u \in V(G), W \subset V(G)$, $$I_2(X_u; X_W\mid Y) \leq \mathrm{conn}_{G,\gamma}(u, W),$$ where $$\gamma((i,j)) = I_2(X_i; X_j\mid Y_{ij})$$ for all $(i,j) \in E(G)$. Here, $\mathrm{conn}_{G,\gamma}(u, W)$ denotes the probability that $u$ is in the same open component as some $w \in W$ in a bond percolation on $G$ where each edge $e$ is independently open with probability $\gamma(e)$.
\end{theorem}

This gives in turn a subadditivity property for the Chi-squared mutual information analogous to Theorem \ref{thm:mainseriesparallelbound}:

\begin{corollary}[from \cite{abbe2018information}]\label{cor:subadditivitybound}
Let $(X,Y)$ be as in Theorem \ref{thm:mainsymmetricbound}. Then for all $u \in V(G), W \subset V(G)$, $I_2$ is subadditive over paths:
\begin{align*}I_2(X_u; X_W\mid Y_{E(G)}) &\leq \sum_{P \in \cP_G(u,v)} I_2(X_u; X_v\mid Y_{E(P)}),\end{align*}
where $\cP_G(u,W)$ is the set of paths (i.e., self-avoiding walks) from $u$ to $W$ in $G$.
\end{corollary}

Notice that when the edge channels are symmetric, the subadditivity bound of Theorem \ref{thm:mainseriesparallelbound} may be weaker than the percolation bound of Theorem \ref{thm:mainsymmetricbound}. However, Theorem \ref{thm:mainseriesparallelbound} also applies to spin synchronization instances with asymmetric channels, while Theorem \ref{thm:mainseriesparallelbound} does not.

For completeness, we include a shortened proof of Theorem \ref{thm:mainsymmetricbound} and Corollary \ref{cor:subadditivitybound} in Appendix~\ref{sec:thmmainsymmetricboundproof}.

\paragraph{Comparison to \cite{yury_percolation}} An  information-percolation bound  is also obtained in \cite{yury_percolation} using an argument based on Strong Data-Processing Inequalities (SDPI). The bound of \cite{yury_percolation} applies to general graphs, to general channels, and also to a model with more general vertex labels. In order to compare the bound of \cite{yury_percolation} with Theorem \ref{thm:mainseriesparallelbound}, we state the result of \cite{yury_percolation} for binary labels:

\begin{theorem}
[Theorem 2 in \cite{yury_percolation} for spins] 
\label{thm:yurytheorem}
Let $(X,Y)$ be a spin synchronization model on a graph $G$. Then, for all $u \in V(G), W \subset V(G)$, $$I_{KL}(X_u; X_W | Y) \leq \mathrm{conn}_{G,\gamma}(u,W),$$ where $$\gamma((i,j)) = \eta(Q_{ij})$$ is the SDPI constant for the edge channel $Q_{ij}$. Again, $\mathrm{conn}_{G,\gamma}(u,W)$ denotes the probability that $u$ is in the same open component as some $w \in W$ in a bond percolation on $G$ where each edge $e$ is independently open with probability $\gamma(e)$.
\end{theorem}
When the channels $Q$ are symmetric in the sense of Theorem \ref{thm:mainsymmetricbound}, the SDPI constant $\eta(Q_{ij})$ equals the Chi-squared mutual information $I_2(X_i; X_j | Y_{ij})$, and therefore Theorem \ref{thm:yurytheorem} almost generalizes Theorem \ref{thm:mainsymmetricbound}, besides for a minor multiplicative factor of 2 resulting from the relationship between $I_{KL}$ and $I_2$ (see Proposition \ref{prop:evanschi2vsklonradhalf}).

However, when the channels $Q$ are asymmetric, $\eta(Q_{ij})$ can be significantly larger than $I_2(X_i; X_j | Y_{ij})$. As a result, our subadditivity bound in Theorem \ref{thm:mainseriesparallelbound} can be used to show that the bound of Theorem \ref{thm:yurytheorem} is loose. We demonstrate this explicitly in Section \ref{sec:generalizedbot} below, by analyzing a spin synchronization problem on a tied-tree model.

\section{Synchronization on trees: a generalization of  broadcasting on trees}

In this section, we first show how the classical broadcasting on tree (BOT) problem of \cite{evans} can be expressed as a synchronization problem on a series-parallel graph consisting of a tied-tree with symmetric channels (Section \ref{sec:botapplication}). Applying our Theorem \ref{thm:mainseriesparallelbound}, we then obtain an alternative proof of the impossibility result of \cite{evans} for BOT reconstruction.

We then generalize this BOT model  (Section \ref{sec:generalizedbot}), using more general channels in the  synchronization on tree (SOT) model. For this generalized BOT, we compare the information bound given by our Theorem \ref{thm:mainseriesparallelbound} based on the Chi-squared subadditivity with the information bound given by Theorem \ref{thm:yurytheorem} from \cite{yury_percolation} based on the SDPI constant.

\subsection{Broadcasting on trees}\label{bot}
In the BOT model of \cite{evans}, a random variable is broadcast from the root down the edges of a tree, with each edge flipping the variable with some probability. The goal is to reconstruct the root variable by observing of the variables of some of its descendants. Formally, each vertex $v \in V(T)$ of a tree $T$ has a binary hidden label $\sigma_v$. The hidden labels are assigned by letting the root $\rho$ have spin $\sigma_{\rho} \sim \Rad(1/2)$, and by defining edge labels $\{\eta_e\} \stackrel{i.i.d.}{\sim} \Rad(\varepsilon)$, and letting $$\sigma_{v} = \sigma_{\rho} \prod_{e} \eta_e,$$ where the product is over the edges in the path from $\rho$ to $v$.

In \cite{evans}, the following information subadditivity inequality is proved, for any finite set of vertices $W \subset V(T)$: \begin{equation}\label{eq:evanssubadditivity}I_{KL}(\sigma_{\rho}; \sigma_W) \leq \sum_{w \in W} I_{KL}(\sigma_{\rho}; \sigma_w).\end{equation} This suffices to show that if $T$ is infinite and has percolation threshold $p_c(T)$, then, when $(1-2\varepsilon)^2 < p_c(T)$, it is impossible to reconstruct $\sigma_{\rho}$ from observations of the leaf variables at infinite depth.


\subsection{Reduction from SOT to BOT}\label{sec:botapplication}

The BOT model with noise parameter $\e$ and tree $T$ that generates the random variables $(\sigma, \eta)$ is equivalent to the synchronization on tree (SOT) model with noise parameter $\e$ and tree $T$ that generates the random variables $(X,Y)$, where the vertex labels $\{X_v\}$ are i.i.d.\ $\Rad(1/2)$, and the edge channels $\{Q_{ij}\}$ are binary symmetric channels (i.e., $Y_{ij} = X_i \cdot X_j \cdot Z_{ij}$, where the $Z_{ij}$ are i.i.d $\Rad(\varepsilon)$). By equivalent, we mean the following.
\begin{proposition}\label{prop:botspinreconstruction}
In the reconstruction setting just described, for all $W \subset V(T)$, \begin{equation*}I(\sigma_{\rho}; \sigma_W) = I(X_{\rho}; X_W, Y_{E(T)}),\end{equation*} where $I$ is any $f$-mutual information (such as $I_{KL}$ or $I_2$).
\end{proposition}
\begin{proof}
First, define $\sigma'_{\rho} = X_{\rho}$ and $\sigma'_w = X_w \cdot \prod_e Y_e$ for all $w \in W$, where the product is over the edges on the path from $\rho$ to $w$. By data-processing and because $\sigma'_{W \cup \{\rho\}} \stackrel{d}{=} \sigma_{W \cup \{\rho\}}$, we have $I(X_{\rho}; X_W, Y) \geq I(\sigma'_{\rho}; \sigma'_W) = I(\sigma_{\rho}; \sigma_W)$.

Conversely, let $X'_{\rho} = \sigma_{\rho}$, $Y'_e \stackrel{i.i.d}{\sim} \Rad(1/2)$, and $X'_{w} = \sigma_{w} \cdot \prod_{e} Y'_e$. By data-processing and equality of distributions, $I(\sigma_{\rho}; \sigma_{W}) = I(\sigma_{\rho}; \sigma_{W}, Y') \geq I(X'_{\rho}; X'_{W}, Y') = I(X_{\rho}, X_W, Y)$.
\end{proof}

Since we wish to apply Theorem \ref{thm:mainseriesparallelbound} in order to bound $I(X_{\rho}; X_W, Y_{E(T)})$, and since we expressed our theorem in terms of  series-parallel graphs with two terminals, we first note that can consider an equivalent varient of SOT using a series-parallel graph by tying the leaves to a terminal with noiseless channels. Namely,  used $(X_{V(T)}, Y_{E(T)})$ to construct a spin synchronization model $(X_{V(G_W)}, Y_{E(G_W)})$ on a series-parallel graph $G_W$, where $W$ are the leaves of the tree $T$, and $G_W$ extends the tree by adding a new vertex $v$ adjacent to all $w \in W$ with noiseless edge channels between any $w$ and $v$. We call $G_W$ a ``tied-tree'' because the leaves $W$ of the tree $T$ are ``tied'' together in $G_W$ through their connections to vertex $v$. 

It is easy to show via induction on the number of edges that tied-trees are indeed series-parallel graphs whose two terminals are the root of the original tree, and the new added vertex connecting the leaves. So the bound of Theorem \ref{thm:mainseriesparallelbound} applied to $G_W$ yields
\allowdisplaybreaks 
\begin{align*}I_2(X_{\rho}; X_v, Y_{E(G_W)}) &= I_2(X_{\rho}; X_v | Y_{E(G_W)}) \tag{Prop. \ref{prop:botspinreconstruction}} \\ &\leq \sum_{P \in \cP_{G_W}(\rho,v)} I_2(X_{\rho}; X_v | Y_{E(P)}) \tag{Theorem \ref{thm:mainseriesparallelbound}} \\ &= \sum_{w \in W} \sum_{P \in \cP_{T}(\rho,w)} I_2(X_{\rho}; X_v | Y_{E(P) \cup (w,v)}) \\ &= \sum_{w \in W} \sum_{P \in \cP_{T}(\rho,w)} I_2(X_{\rho}; X_w | Y_{E(P) \cup (w,v)}) \tag{Since $Y_{(w,v)}$ is noiseless} \\ &= \sum_{w \in W} \sum_{P \in \cP_{T}(\rho,w)} I_2(X_{\rho}; X_w | Y_{E(P)}) \tag{Since $Y_{(w,v)} \indy X_{\rho}, X_w, Y_{E(P)}$} \\ &= \sum_{w \in W} \sum_{P \in \cP_{T}(\rho,w)} I_2(X_{\rho}; X_w | Y_{E(T)}) \tag{Since $Y_{E(T) \sm E(P)} \indy Y_{E(P)}, X_{\rho}, X_w$} \\ &= \sum_{w \in W} \sum_{P \in \cP_{T}(\rho,w)} I_2(X_{\rho}; X_w, Y_{E(T)}) \tag{Prop. \ref{prop:i2conditionequalsi2side}} \\ &= \sum_{w \in W} I_2(\sigma_{\rho}; \sigma_w) \tag{By Prop. \ref{prop:botspinreconstruction} and uniqueness of paths in tree}  \end{align*}

We conclude by noting that since $Y_{\{(w,v)\}_{w \in W}} \indy X_{\rho} | X_W, Y_{E(T)}$, $$I_2(X_{\rho}; X_v, Y_{E(G_W)}) = I_2(X_{\rho}; X_W, Y_{E(T)}, Y_{\{(w,v)\}_{w\in W}}) = I_2(X_{\rho}; X_W, Y_{E(T)}) = I_2(\sigma_{\rho}; \sigma_W).$$

Hence, we have re-derived the BOT subadditivity inequality \eqref{eq:evanssubadditivity} for the Chi-squared mutual information (and also for the KL mutual information with a factor of 2 by Proposition \ref{prop:evanschi2vsklonradhalf}). This in turn gives the impossibility result of \cite{evans}.

\subsection{Asymmetric SOT and Chi-squared v.s.\ SDPI}\label{sec:generalizedbot}
In Section \ref{sec:botapplication}, we formulated the spin broadcasting on trees problem of \cite{evans} as a spin synchronization problem on a tied-tree graph, with binary symmetric edge channels. If we relax the edge channels to allow for general (possibly asymmetric) channels, we thus obtain a generalization of the broadcasting on trees problem. Notice that the path subadditivity bound of Theorem \ref{thm:mainseriesparallelbound}  still applies to this generalized model, as the tied-tree is series-parallel.  

In particular, suppose we are given a spin synchronization model $(X,Y)$ on a tied-tree $G$ with root $u$ and terminal $v$ similarly to previous section, but we now take the edge channels $Q$ on the underlying tree to be asymmetrical, i.e., $Q_{ij}(\cdot | +1) \sim \Ber(a/n)$ and $Q_{ij}(\cdot |-1) \sim \Ber(b/n)$, and the leaves of the tree are tied together at $v \in V(G)$ with noiseless edge channels. Suppose moreover that the tree is a regular tree with $d=n$ descendants at each generation and depth $t$. 

For any edge $(i,j)$, we have $$I_2(X_i; X_j | Y_{ij}) = \frac{(a-b)^2}{2(a+b)n} + \frac{(a-b)^2}{n^2(1-(a+b)/(2n))} = \frac{(a-b)^2}{2(a+b)n} + o(1/n)$$ by direct calculation using Proposition \ref{prop:correqualschi2}. Hence, by Theorem \ref{thm:mainseriesparallelbound}, we have the bound \begin{align}I_2(X_u; X_v | Y_{E(G)}) &\leq \sum_{P \in \cP_{G}(u,v)} I_2(X_u; X_v | Y_{E(P)}) \tag*{} \\ &= \sum_{P \in \cP_{G}(u,v)} \left(\frac{(a-b)^2 + o(1)}{2(a+b)n}\right)^{|E(P)|}\label{eq:ourresontiedtree}\\
&=d^t \left(\frac{(a-b)^2}{2(a+b)n} + o(1/n)\right)^{t}\\
&=\left(\frac{(a-b)^2}{2(a+b)} + o_n(1)\right)^{t}\label{last1}
.\end{align}

On the other hand, as derived in \cite{polyanskiy2017strong}, the SDPI constant for the edge channels is $$\eta(Q_{ij}) = \frac{(\sqrt{a}-\sqrt{b})^2 }{n}  + o(1/n).$$ So letting $\gamma((i,j)) = \eta(Q_{i,j})$ for all $(i,j) \in E(G)$, the bound of \cite{yury_percolation} (Theorem \ref{thm:yurytheorem}) gives \begin{align}I_2(X_u; X_v | Y_{E(G)}) &\leq \mathrm{conn}_{G,\gamma}(u,v) \label{eq:yuryontiedtree} \end{align} And by inclusion-exclusion, \begin{align}\mathrm{conn}_{G,\gamma}(u,v) &\geq \sum_{P \in \cP_{G}(u,v)} \gamma^{|E(P)|} - \sum_{P_1\neq P_2 \in \cP_{G}(u,v)} \gamma^{|E(P_1) \cup E(P_2)|} \\ &= d^t \gamma^t - \frac{d^t \gamma^t}{2} \cdot \left(\sum_{k=1}^t\gamma^k (d-1) \cdot d^{k-1}\right) \\ &= d^t \gamma^t \left(1 - \frac{(d-1)(1-(\gamma d)^{t+1})}{2d(1-\gamma d)}\right) \\
&\geq \left((\sqrt{a}-\sqrt{b})^2 + o_n(1)\right)^{t} \left(1 - \frac{1}{2}\left((\sqrt{a}-\sqrt{b})^2 + o_n(1)\right)^{t+1}\right)
\end{align}
Note that for any $a,b>0$, 
$$ (\sqrt{a} - \sqrt{b})^2 \ge \frac{(a-b)^2}{2(a+b)},$$ (and the inequality is strict when $a \neq b$), so when $(\sqrt{a}-\sqrt{b})^2 < 1$ and $t$ and $n$ are sufficiently large, the bound \eqref{eq:ourresontiedtree} derived from Theorem \ref{thm:mainseriesparallelbound} is tighter than the bound \eqref{eq:yuryontiedtree} implied by Theorem \ref{thm:yurytheorem} from \cite{yury_percolation}.

Further, applying a union bound on the paths in order to upper-bound the connection probability in the percolation of \cite{yury_percolation}, we have:
\begin{align}I_2(X_u; X_v | Y_{E(G)}) &\leq \mathrm{conn}_{G,\gamma}(u,v) \\ &\leq \sum_{P \in \cP_G(u,v)} \gamma^{|E(P)|} \\ &=
d^t \cdot \gamma^t \\ &= \left((\sqrt{a}-\sqrt{b})^2 + o_n(1)\right)^t \label{last22}
\end{align}
Therefore, taking now $$(\sqrt{a}-\sqrt{b})^2>1>\frac{(a-b)^2}{2(a+b)},$$
and $t$ such that $t=O(n)$ and $t = \omega_n(1)$ (e.g., $t=n$), we have that \eqref{last1} vanishes as $n$ diverges, and our bound implies that $u$ and $v$ cannot be synchronized non-trivially (a.k.a.\ reconstruction/weak recovery is not solvable). In contrast, \eqref{last22} blows up as $n$ diverges, i.e., \eqref{last22} tends to 1 (as one can always cap the upper-bound to 1).

\section{Proof of Theorem \ref{thm:mainseriesparallelbound}}
The result is obtained by using (i) a multiplicative property of the Chi-squared mutual information on paths (see Proposition \ref{prop:chi2seriesmultiplicativity}), and (ii) a subadditivity property of the Chi-squared mutual information on depth-1 trees (see Lemma \ref{lem:chi2parallelsubadditivity}). Interestingly, (i) does not hold for the classical mutual information, $I_{KL}$, making the Chi-squared mutual information, $I_2$, a natural choice for this proof:

\begin{proof}[Proof (of Theorem \ref{thm:mainseriesparallelbound})]
In the following, we implicitly use $I_2(X_kX_l; Y) = I_2(X_k; X_l | Y)$, by Proposition \ref{prop:i2equalsprod}.

The proof is by induction on $|E(G)|$. The base case, $|E(G)| = 1$, is trivial.
For the inductive step, one of two cases holds: 

(Case 1) $G$ is the series composition of $H_1$ which is series-parallel with terminals $u,w$, and $H_2$, which is series-parallel with terminals $w,v$. \begin{align*}&I_2(X_u X_v ; Y_{E(G)}) = I_2((X_u X_w) \cdot (X_w X_v) ; Y_{E(H_1)}, Y_{E(H_2)}) \\ &= I_2(X_u X_w; Y_{E(H_1)}) I_2(X_w X_v; Y_{E(H_2)}) \tag*{(Prop. \ref{prop:chi2seriesmultiplicativity})} \\ &\leq \textstyle\sum_{\substack{P_1 \in \cP_{H_1}(u,w) \\ P_2 \in \cP_{H_2}(w,v)}} I_2(X_uX_w; Y_{E(P_1)}) I_2(X_wX_v; Y_{E(P_2)}) \\ &= \textstyle\sum_{\substack{P_1 \in \cP_{H_1}(u,w) \\ P_2 \in \cP_{H_2}(w,v)}} I_2(X_uX_v; Y_{E(P_1)}, Y_{E(P_2)}) \tag*{(Prop. \ref{prop:chi2seriesmultiplicativity})} \\ &= \textstyle \sum_{P \in \cP_{G}(u,v)} I_2(X_uX_v; Y_{E(G)}). \end{align*}
The inequality is by the inductive hypothesis.

(Case 2) $G$ is the parallel composition of $H_1$ and $H_2$ both series-parallel, with terminals $u,v$.
Then, \begin{align*}I_2&(X_uX_v; Y_{E(G)}) = I_2(X_uX_v; Y_{E(H_1)}, Y_{E(H_2)}) \\ &\leq I_2(X_uX_v; Y_{E(H_1)}) + I_2(X_uX_v; Y_{E(H_2)}) \tag*{(Lem. \ref{lem:chi2parallelsubadditivity})} \end{align*}
The inductive step follows by the inductive hypothesis, since $\cP_G(u,v) = \cP_{H_1}(u,v) \sqcup \cP_{H_2}(u,v)$.
\end{proof}

\section{Future directions}\label{sec:futuredirections}

\paragraph{Connections with correlation decay} As mentioned in the introduction, fixing the edge observations and applying Ising model correlation decay conditions yields bounds that are not as strong as those we proved in this paper, because the techniques in our paper allow us to deal with the average-case edge observations, while fixing the edge observations and applying the Dobrushin conditions requires us to work with the worst-case edge observations. It would nonetheless be interesting to elaborate on this connection. 


\paragraph{Information subadditivity for general graphs}
A possible extension would be to prove that Theorem \ref{thm:mainseriesparallelbound} applies to spin synchronization models on general graphs $G$, that have general asymmetric edge channels.
\begin{conjecture}[Generalization of Theorem \ref{thm:mainseriesparallelbound}] \label{conj:mainseriesparallelboundgeneralization}
Let $(X,Y)$ be a spin synchronization model on {\em any} graph $G$, and arbitrary edge channels $Q$. Then, for any $u,v \in V(G)$, $$I_2(X_u; X_v \mid Y_{E(G)}) \leq \sum_{P \in \cP_G(u,v)} I_2(X_u; X_v \mid Y_{E(P)}).$$ Here $\cP_G(u,v)$ is the set of paths (self-avoiding walks) from $u$ to $v$ in $G$.
\end{conjecture}
In particular, Conjecture \ref{conj:mainseriesparallelboundgeneralization} would  directly imply the two-community stochastic block-model information-theoretic threshold for impossibility of weak recovery, without having to reduce to the broadcasting on trees model.

Another related subadditivity bound that one could hope for is: 
\begin{conjecture}\label{conj:conjecturedsubadditivityvariation} Let $(X,Y)$ be a spin synchronization model on a graph $G$, and arbitrary edge channels $Q$. Then, for any $u \in V(G)$, $W \subset V(G)$,
\begin{equation*}I(X_u; X_W | Y_{E(G)}) \leq \sum_{w \in W} I(X_u; X_w | Y_{E(G)}),\end{equation*} where $I$ is either the Chi-squared mutual information, $I_2$, or the KL mutual information, $I_{KL}$.
\end{conjecture}
Applied to spin synchronization models on trees, Conjecture \ref{conj:conjecturedsubadditivityvariation} would (also) directly imply the information subadditivity inequality \eqref{eq:evanssubadditivity} from \cite{evans} for broadcasting on trees.

\paragraph{Information subadditivity for non-uniform binary spins}
The information subadditivity inequality of Theorem \ref{thm:mainseriesparallelbound} does not hold when the spin synchronization model is generalized to allow for non-uniform binary spins.

For example, consider a graph $G$ with two nodes $u,v$, and two parallel edges $e,f$ between the nodes. Suppose we have vertex labels $X_u, X_v \stackrel{i.i.d}{\sim} \Rad(\delta)$, and edge labels $Y_e = X_uX_v Z_e$ and $Y_f = X_uX_v Z_f$, where $Z_e,Z_f \stackrel{i.i.d} \sim \Rad(\epsilon)$. Then, by direct calculation
$$I_2(X_u; X_v | Y_e, Y_f) = \frac{\delta^2(1-\delta)^2(1-2\eps)^2 (\eps^2 + (1-\eps)^2)^3}{(-4\delta^2\eps^2+4\delta^2\eps-\delta^2+4\delta\eps^2-4\delta\eps+\delta+\eps^4-2\eps^3+\eps^2)^2},$$ and $$I_2(X_u; X_v | Y_e) = I_2(X_u; X_v | Y_f) = \frac{\delta^2(1-\delta)^2(1-2\eps)^2}{(-4\delta^2\eps^2+4\delta^2\eps-\delta^2+4\delta\eps^2-4\delta\eps+\delta-\eps^2+\eps)^2}.$$

When $\delta = 0.2$ and $\eps = 0.2$, one may check that $I_2(X_u; X_v | Y_e, Y_f) \geq I_2(X_u; X_v | Y_e) + I_2(X_u; X_v | Y_f)$, contradicting subadditivity.

This leaves open the question of whether there exists a natural generalization of the subadditivity theorem that holds when the vertex labels are non-uniform (possibly with additional factors). 

\paragraph{Information subadditivity for general alphabets} Another way to generalize the vertex labels of the spin synchronization model is to draw them from larger alphabets than the binary alphabet. For example, as in \cite{our_grid}, one may consider vertex labels $X_v$ that are uniformly chosen from a group $\cG$, and edge labels $Y_{uv}$ that are noisy observations of the difference of the endpoints, $X_vX_u^{-1}$.

However, the subadditivity of the Chi-squared information over paths does not extend in a direct way to these larger alphabets. For example, consider a ``spoon''-shaped graph $G$ with vertices $u,v,w$, an edge $e$ between $u$ and $v$, and two parallel edges $f_1,f_2$ between $v$ and $w$. Let the vertex labels $X_u,X_v,X_w$ be i.i.d uniform in $\ZZ/4\ZZ$, and let the edge labels be as follows:
$$Y_e = \begin{cases} 0, & \mbox{if } X_u - X_v \in \{0,1\} \\ 1, & \mbox{if } X_u - X_v \in \{2,3\} \end{cases}, Y_{f_1} = \begin{cases} 0, & \mbox{if } X_v - X_w \in \{0,1\} \\ 1, & \mbox{if } X_v - X_w \in \{2,3\} \end{cases},Y_{f_2} = \begin{cases} 0, & \mbox{if } X_v - X_w \in \{0,2\} \\ 1, & \mbox{if } X_v - X_w \in \{1,3\} \end{cases}.$$

Notice that $I_2(X_u; X_w | Y_e, Y_{f_1}, Y_{f_2}) = I_2(X_u; X_v | Y_e) = 1$. Also, $I_2(X_u; X_w | Y_e, Y_{f_1}) = 1/2$ and $I_2(X_u; X_w | Y_e, Y_{f_2}) = 0$. Therefore, the subadditivity over paths $I_2(X_u; X_w | Y_e, Y_{f_1}, Y_{f_2}) \leq I_2(X_u; X_w | Y_e, Y_{f_1}) + I_2(X_u; X_w | Y_e, Y_{f_2})$ does not hold.

Nevertheless, it is an interesting problem to find a generalization of information subadditivity when the vertex labels are uniform over groups. For example, it is plausible that the scaling of the information-theoretic threshold for weak recovery in the $k$-community SBM could be recovered from such a generalization.

\paragraph{A converse to Theorem \ref{thm:mainsymmetricbound}}
When the edge channels of the spin synchronization model are symmetric, Theorem \ref{thm:mainsymmetricbound} from \cite{abbe2018information} and Theorem \ref{thm:yurytheorem} from \cite{yury_percolation} are tight on trees, so one cannot open the edges with lower probability in general. Is there a converse to Theorem \ref{thm:mainsymmetricbound}: i.e., is the mutual information lower-bounded by the connection probability on some non-trivial percolation? For example, for some bounded-degree graphs?

\appendices

\section{Properties of Chi-squared mutual information}\label{app:chi2mutualinformation}

\begin{definition} For two random variables $A,B$ with joint distribution $\nu_{A,B}$, the Chi-squared mutual information $I_2(A;B)$ is the $f$-mutual information $I_f(A; B) := D_f(\nu_{A,B} || \nu_A \times \nu_B)$ for the choice $f(t) = (t-1)^2$. The KL mutual information is the $f$-mutual information with the choice $f(t) = t \log_2 t$. Here, $D_f(\mu || \nu) := \int f\left(\frac{d\mu}{d\nu}\right) d\nu$ is the $f$-divergence introduced in \cite{csiszar1967information}, which is well-defined and has desirable properties such as nonnegativity and monotonicity when $\mu \ll \nu$, $f$ is convex, $f$ is strictly convex at 1, and $f(1) = 0$. In particular, the $f$-mutual informations have a data processing inequality.
\end{definition}

\begin{proposition}\label{prop:correqualschi2}
Let $A,U$ be jointly-distributed random variables, with $U \in \{-1,+1\}$. Then $$I_2(A; U) = \frac{\Var[\EE[U | A]]}{\Var[U]}.$$ In particular, if $U \sim \Rad(1/2)$, then $$I_2(A; U) = \EE[\EE[U | A]^2].$$
\end{proposition}
\begin{proof}
Letting $\nu_Z$ denote the distribution of $Z$, and $\Omega$ denote the sample set of $A$,
\begin{align}
        I_2(A; U) &= \int_{\Omega \times \{-1,+1\}} \left(\frac{d(\nu_{A,U})}{d(\nu_A \times \nu_U)} - 1\right)^2 d(\nu_{A} \times \nu_{U}) \tag*{} \\ &= \int_{\Omega} \sum_{u \in \{-1,+1\}} \nu_U(u) \left(\frac{1}{ \nu_U(u)} \cdot \frac{d(\nu_{A,U}(\cdot, u))}{d\nu_A(\cdot)} - 1\right)^2 d\nu_A \tag*{} \\ &= \int_{\Omega} \sum_{u \in \{-1,+1\}} \frac{1}{ \nu_U(u)} \left(\frac{d(\nu_{A,U}(\cdot, u))}{d\nu_A(\cdot)} - \nu_U(u)\right)^2 d\nu_A \tag*{}
\end{align}
So, since $$\frac{d(\nu_{A,U}(\cdot, 1))}{d\nu_A(\cdot)} - \nu_U(1) = \left(1 - \frac{d(\nu_{A,U}(\cdot, -1))}{d\nu_A(\cdot)}\right) - (1 - \nu_U(-1)) = -\left(\frac{d(\nu_{A,U}(\cdot, -1))}{d\nu_A(\cdot)} - \nu_U(-1)\right)$$ $\nu_A$-almost everywhere, we have 
\begin{align*}
 I_2(A; U) &= \left( \sum_{u \in \{-1,+1\}}\frac{1}{ \nu_U(u)} \right) \cdot \int_{\Omega}  \left(\frac{d(\nu_{A,U}(\cdot, 1))}{d\nu_A(\cdot)} - \nu_U(1)\right)^2 d\nu_A \\ &= \frac{4}{\Var[U]} \cdot \int_{\Omega} \left(\PP[U = 1 | A] - \PP[U = 1]\right)^2 d\nu_A \\ &= \frac{4\Var[\PP[U = 1|A]]}{\Var[U]} = \frac{4\Var[\EE[(U/2) | A]]}{\Var[U]} = \frac{\Var[\EE[U|A]]}{\Var[U]}.
\end{align*}
When $U \sim \Rad(1/2)$, we have $\EE[U] = 0$ and $\Var[U] = 1$, so $I_2(A; U) = \EE[\EE[U | A]^2]$.
\end{proof}

\begin{proposition}\label{prop:i2equalsprod} Let $(X,Y)$ be a spin synchronization model. Then $I_2(X_uX_v; Y) = I_2(X_u; X_v, Y)$.
\end{proposition}
\begin{proof} \begin{align*}I_2(X_u; X_v, Y) &= I_2(X_uX_v; X_v, Y) \tag{Data-processing} \\ &= I_2(X_uX_v; Y) \tag{Since $X_uX_v, Y \indy X_v$}.\end{align*}
\end{proof}

\begin{proposition}\label{prop:i2conditionequalsi2side}Let $(X,Y)$ be a spin synchronization model. Then $I_2(X_u; X_W, Y) = I_2(X_u; X_W | Y).$
\end{proposition}
\begin{proof}
\begin{align*}I_2(X_u; X_W, Y) &= \EE[\EE[X_u | X_W, Y]^2] \tag{Prop. \ref{prop:correqualschi2}} \\ &= \EE[\EE[\EE[X_u | X_W, Y]^2 | Y]] \\ &= I_2(X_u; X_W | Y) \tag{Prop. \ref{prop:correqualschi2}, since $X_u \indy Y$, so $X_u | Y \sim \Rad(1/2)$} \end{align*}
\end{proof}

\begin{proposition}\label{prop:evanschi2vsklonradhalf}
Let $A,U$ be joint random variables, $U \sim \Rad(1/2)$. Then $$\frac{1}{2}I_2(A; U) \leq I_{\mathrm{KL}}(U; A) \leq I_2(U; A),$$ where $I_{KL}$ is in bits.
\end{proposition}
\begin{proof}
As shown in \cite{evans}, this follows from the inequalities $$\frac{x^2}{2} \leq \frac{1+x}{2} \log_2(1-x) + \frac{1-x}{2} \log_2(1+x) \leq x^2.$$
\end{proof}

\begin{proposition}\label{prop:chi2seriesmultiplicativity}
Let $U,V,W \stackrel{i.i.d}{\sim} \Rad(1/2)$. Let $A$ be the output of a channel on $UW$, and let $B$ be the output of a channel on $WV$. Then $$I_2( UV; A,B) = I_2(UW; A)I_2(WV; B).$$
\end{proposition}
\begin{proof}
\begin{align*}I_2(UV; A,B) &= \EE[\EE[UV|A,B]^2] \tag{Prop. \ref{prop:correqualschi2}} \\ &= \EE[\EE[UWWV| A,B]^2] \tag{Since $W^2 = 1$}\\
&= \EE[\EE[UW|A,B]^2\EE[WV|A,B]^2] \tag{Using $UW \indy WV | A,B$}\\
&= \EE[\EE[UW | A]^2\EE[WV | B]^2] \tag{Using $UW \indy B | A$ and $WV \indy A | B$} \\ &= \EE[\EE[UW | A]^2] \EE[\EE[WV | B]^2] \tag{Using $A \indy B$ because $U,V \stackrel{i.i.d}{\sim} \Rad(1/2)$.} \\ &= I_2(UW; A)I_2(WV; B) \tag{Prop. \ref{prop:correqualschi2}}.\end{align*}
\end{proof}

\begin{lemma}\label{lem:chi2parallelsubadditivity}
Let $U \sim \Rad(1/2)$. Let $A$ and $B$ be the outputs of two independent channels on $U$. Then $I_2(U; A,B) \leq I_2(U; A) + I_2(U; B).$
\end{lemma}
\begin{proof}Let $\nu_{A,B,U}$ denote the joint distribution of $A,B,U$. For simplicity, we prove the lemma when $A,B$ are discrete. For any two random variables $C,D$ with joint distribution $\nu_{C,D}$, $I_2(C; D) = D_{(1-1/t)}(\nu_C\nu_D||\nu_{C,D})$, where $D_{(1/1-t)}$ is the $(1-1/t)$-divergence, so\begin{align}
&I_2(U; A,B) - (I_2(U; A) + I_2(U; B)) \tag*{} \\ &= \EE[(\frac{\nu_{A,B,U}}{\nu_{A,B}\nu_U}-1) - (\frac{\nu_{A,U}}{\nu_A\nu_U}-1)-(\frac{\nu_{B,U}}{\nu_B\nu_U}-1)] \tag*{} \\ &= \EE[(\frac{\nu_{U|A,B}}{\nu_{U|A}}-1)(\frac{\nu_{A,U}}{\nu_A\nu_{U}} - 1)] \label{eq:term1} \\
&~~~~~~~+\EE[(\frac{\nu_{U|A,B}}{\nu_{U|A}}-1) - (\frac{\nu_{B,U}}{\nu_{B}\nu_{U}} - 1)] \label{eq:term2}
\end{align}


We claim Terms \eqref{eq:term1} and \eqref{eq:term2} are $\leq 0$, which implies the lemma statement.

We rewrite Term \eqref{eq:term1}, using $\frac{\nu_{A,U}(a,u)}{\nu_A(a)\nu_U(u)}-1 = 2 \nu_{U|A}(u|a) - 1 = \EE[U|A=a] \cdot u$:
\begin{align*}
    \mbox{\eqref{eq:term1}} &= \EE[(\frac{\nu_{U|A,B}}{\nu_{U|A}} - 1) \cdot \EE[U|A] \cdot U] \\ &= \EE[\EE[U|A] \cdot (U \cdot \frac{\nu_{U|A,B}}{\nu_{U|A}} - U)] \\ &= \EE[\EE[U|A] \cdot (U\cdot \frac{\nu_{U|A,B}}{\nu_{U|A}} - \EE[U|A])].
\end{align*}
Define $$t_a := \sum_{b} \frac{\nu_{B,U|A}(b,1|a)\nu_{B,U|A}(b,-1|a)}{\nu_{B,U|A}(b,1|a) + \nu_{B,U|A}(b,-1|a)}.$$

Note \begin{equation}\label{ineq:subadditivityofteaay}0 \leq t_a \leq \nu_{U|A}(1|a) \nu_{U|A}(-1|a)\end{equation} by the subadditivity of $f(x,y) = xy/(x+y)$ for $x,y \geq 0$. (In particular, for all $a,b,c,d \geq 0$, $f(a,b) + f(c,d) \leq f(a+c,b+d)$.) So

\begin{align*}&\EE[U \cdot \frac{\nu_{U|A,B}}{\nu_{U|A}} | A] = \sum_{u} \frac{u}{\nu_{U|A}} \sum_{b} \frac{\nu_{B,U|A} \nu_{B,U|A}}{\nu_{B|A}} \\ &= \sum_{u} \frac{u}{\nu_{U|A}} \sum_{b} \left(\frac{(\nu_{B|A} - (\nu_{B|A} -\nu_{B,U|A}))\nu_{B,U|A}}{\nu_{B|A}}\right) \\ &= \sum_{u} \frac{u}{\nu_{U|A}} (-t_A + \sum_{b} \nu_{B,U|A}) = \sum_{u} \frac{u}{\nu_{U|A}} (\nu_{U|A} - t_A) \\ &= -\sum_{u} u \frac{t_A}{\nu_{U|A}} \\ &= -\frac{t_A}{\nu_{U|A}(1|A)} + \frac{t_A}{\nu_{U|A}(-1|A)} \\ &= \frac{t_A (\nu_{U|A}(1|A) - \nu_{U|A}(-1|A))}{\nu_{U|A}(1|A)\nu_{U|A}(-1|A)} \\ &=  \frac{t_A}{\nu_{U|A}(1|A)\nu_{U|A}(-1|A)} \cdot \EE[U|A] \\
&= c_A\EE[U|A],\end{align*} for some $0 \leq c_A \leq 1$ by \eqref{ineq:subadditivityofteaay}. Thus, $\eqref{eq:term1} = \EE[U|A]^2 (c_A - 1) \leq 0$, as desired.

Now we bound Term \eqref{eq:term2}.
\begin{align*}
    \mbox{\eqref{eq:term2}} &= \EE[(\frac{\nu_{B,U}}{\nu_{B|A}\nu_U} - 1) - (\frac{\nu_{B,U}}{\nu_B\nu_{U}} - 1)] \tag{Using $B \indy A | U$}\\ &= \EE[(\frac{\nu_A\nu_{B}}{\nu_{A,B}}-1)(\frac{\nu_{B,U}}{\nu_B\nu_U} - 1)] \tag{Since $\EE[\frac{\nu_B}{\nu_{B|A}} - 1] = 0$} \\
&= \sum_{a,b,u} \nu_{A,B,U} \left(\frac{\nu_{B}}{\nu_{B \mid A}} - 1\right)\left(\frac{\nu_{B \mid U}}{\nu_{B}} - 1\right) 
\end{align*}
For compactness, write $\alpha_a = \nu_{A|U}(a|1), \beta_a = \nu_{A|U}(a|-1), \gamma_b = \nu_{B|U}(b|1), \delta_b = \nu_{B|U}(b|-1)$:
\begin{align*}
\eqref{eq:term2} &= \sum_{a,b} \frac{\alpha_a\gamma_b}{2} \left(\frac{(\gamma_b + \delta_b)/2}{(\alpha_a\gamma_b + \beta_a\delta_b)/(\alpha_a + \beta_a)} - 1\right)\left(\frac{\gamma_b}{(\gamma_b + \delta_b)/2} - 1\right) \\ &~~~~~~~~~~~+ \frac{\beta_a\delta_b}{2} \left(\frac{(\gamma_b + \delta_b)/2}{(\alpha_a\gamma_b + \beta_a\delta_b)/(\alpha_a + \beta_a)} - 1\right)\left(\frac{\delta_b}{(\gamma_b + \delta_b)/2} - 1\right) \\
&= \sum_{a,b} \frac{\alpha_a\gamma_b}{2} \left(\frac{(\gamma_b + \delta_b)(\alpha_a + \beta_a)}{2(\alpha_a\gamma_b + \beta_a\delta_b)} - 1\right)\left(\frac{\gamma_b - \delta_b}{\gamma_b + \delta_b}\right) + \frac{\beta_a\delta_b}{2} \left(\frac{(\gamma_b + \delta_b)(\alpha_a + \beta_a)}{2(\alpha_a\gamma_b + \beta_a\delta_b)} - 1\right)\left(\frac{\delta_b - \gamma_b}{\gamma_b + \delta_b}\right) \\
&= \sum_{a,b}  \left(\frac{\gamma_b - \delta_b}{\gamma_b + \delta_b}\right)\left(\frac{\alpha_a\gamma_b}{2}\left(\frac{(\gamma_b + \delta_b)(\alpha_a + \beta_a)}{2(\alpha_a\gamma_b + \beta_a\delta_b)} - 1\right) - \frac{\beta_a\delta_b}{2} \left(\frac{(\gamma_b + \delta_b)(\alpha_a + \beta_a)}{2(\alpha_a\gamma_b + \beta_a\delta_b)} - 1\right)\right) \\ 
&= \sum_{a,b} \left(\frac{\gamma_b - \delta_b}{\gamma_b + \delta_b}\right)\left(\frac{\alpha_a\gamma_b - \beta_a\delta_b}{2}\right) \left(\frac{(\gamma_b + \delta_b)(\alpha_a + \beta_a)}{2(\alpha_a\gamma_b + \beta_a\delta_b)} - 1\right) \\
&= \sum_{a,b} \left(\frac{\gamma_b - \delta_b}{\gamma_b + \delta_b}\right)\left(\frac{\alpha_a\gamma_b - \beta_a\delta_b}{4}\right) \left(\frac{-(\alpha_a - \beta_a)(\gamma_b - \delta_b)}{\alpha_a\gamma_b + \beta_a\delta_b}\right) \\
&= \sum_{a,b} \left(\frac{(\gamma_b - \delta_b)^2}{4(\gamma_b + \delta_b)}\right) \left(-(\alpha_a - \beta_a)\frac{\alpha_a\gamma_b - \beta_a\delta_b}{\alpha_a\gamma_b + \beta_a\delta_b}\right) \\
&= \sum_{b} \left(-\frac{(\gamma_b - \delta_b)^2}{4(\gamma_b + \delta_b)}\right) \sum_{a} \left((\alpha_a - \beta_a)\frac{\alpha_a\gamma_b - \beta_a\delta_b}{\alpha_a\gamma_b + \beta_a\delta_b}\right).
\end{align*} We conclude by using \begin{equation*}\sum_a (\alpha_a - \beta_a) = \sum_a \nu_{A|U}(a|1) - \nu_{A|U}(a|-1) = 1 - 1 = 0,\end{equation*} so \begin{align*}\sum_a (\alpha_a - \beta_a) (\frac{\alpha_a \gamma_b - \beta_a \delta_b}{\alpha_a \gamma_b + \beta_a \delta_b}) &= \sum_a (\alpha_a - \beta_a) (\frac{\alpha_a \gamma_b - \beta_a \delta_b}{\alpha_a \gamma_b + \beta_a \delta_b} + 1) \\ &= \sum_a (\alpha_a - \beta_a) (\frac{2\alpha_a\gamma_b}{\alpha_a \gamma_b + \beta_a \delta_b}) \\ &\geq \sum_a (\alpha_a - \beta_a) (\frac{2\gamma_b}{\gamma_b + \delta_b}) \tag{The inequality is term-wise} \\ &= 0.\end{align*}

Since $\left(-\frac{(\gamma_b - \delta_b)^2}{4(\gamma_b + \delta_b)}\right) \leq 0$ for all $b$, this proves that Term $\eqref{eq:term2} \leq 0$.

\end{proof}


\section{Proof of Theorem \ref{thm:mainsymmetricbound} and Corollary \ref{cor:subadditivitybound}}\label{sec:thmmainsymmetricboundproof}
For the sake of completeness, we provide here a compact proof of the main result of \cite{abbe2018information}.

\begin{proof}[Proof (of Theorem \ref{thm:mainsymmetricbound})]
~
\\
\indent \textit{(Step 1) Reduce to case $|W| = 1$}: Construct the graph $G'$ by adding a new vertex $v$ adjacent to all $w \in W$, letting $X_v \sim \Rad(1/2)$ and defining $Y_{vw} = X_vX_w$ for all $w \in W$. Now $I_2(X_u; X_v | Y_{E(G')}) = I_2(X_u; X_W | Y_{E(G')}) = I_2(X_u; X_v | Y_{E(G)})$, $\mathrm{conn}_{G',\gamma'}(u,v) = \mathrm{conn}_{G,\gamma}(u,W)$, and $(X,Y)$ is a symmetric spin synchronization instance on $G'$.
Hence, it suffices to prove the bound 
$I_2(X_u; X_W | Y) \leq \mathrm{conn}_{G,\gamma}(u,W)$ in the case $|W| = 1$.

\textit{(Step 2) Reduce to BSC case}:
For each edge channel $Q_e(\cdot | \pm 1)$, let $T_e$ be the symmetry transformation and define $Z_e = \{Y_e, T_e(Y_e)\}$. By the symmetry property of the edge channels, $X \indy Z$. So $\cL(X,Y|Z)$, the law of $(X,Y)$ conditioned on $Z$, is almost surely the law of a spin synchronization model $(X',Y')$ on $G$, where each of the channels $Q'_e$ is binary-valued (either $Y_e$ or $T_e(Y_e)$). Explicitly, for $z \in Z_e$, \begin{align*}Q'_e(z | +1) &= \frac{dQ_e(z | +1)}{d(Q_e(z | +1) + Q_e(T_e(z) | +1))},\end{align*} and by the symmetry property this equals \begin{align*}\frac{dQ_e(T_e(z) | -1)}{d(Q_e(T_e(z) | -1) + Q_e(z | -1))} = Q'_e(T_e(z) | -1).\end{align*} So $Q'_e$ is a binary symmetric channel.
Hence, proving Theorem \ref{thm:mainsymmetricbound} when all the edge channels are BSC yields the general bound: \begin{align*}I_2(X_u; X_v | Y) &= \EE_Z[I_2(X_u; X_v | Y,Z)] \\ &= \EE_Z[I_2(X'_u; X'_v | Y', Z)] \\ &\leq \EE_Z[\mathrm{conn}_{G,\gamma_Z}(u,v)] \tag*{(by BSC case)} \\ &= \mathrm{conn}_{G,\gamma}(u,v) \hspace{1.42cm} \mbox{(since $\gamma = \EE_Z[\gamma_Z]$)}.\end{align*} Here, $\gamma_Z((i,j)) = I_2(X_i, X_j | Y_{ij},Z_{ij}),$ and for the last equality we use the fact that $\gamma((i,j)) = I_2(X_i, X_j | Y_{ij}) = \EE_Z[\gamma_Z((i,j))]$.

\textit{(Step 3) Prove BSC case}: Now it only remains to prove $I_2(X_u; X_v | Y) \leq \mathrm{conn}_{G,\gamma}(u,v)$ when all edge channels $Q_e$ are BSC. Let the flip probability of $Q_e$ be $\varepsilon(e)$, and define $\delta(e) = (1-2\varepsilon(e))$. We can assume that $\delta(e) \in [0,1]$, because we lose no information by flipping edge labels deterministically. Also, by direct calculation $\gamma(e) = \delta(e)^2.$

The proof goes by induction on $|S_{\delta}|$, where $$S_{\delta} := \{e \in E(G) : \delta(e) \not\in \{0,1\}\}.$$

In the base case, $|S_{\delta}| = 0$, so all edge observations are completely noiseless or completely noisy. Hence, $I_2(X_u, X_v | Y) = 1$ if there is a path $P$ from $u$ to $v$ whose edges are all noiseless. If there is no such path, then $I_2(X_u, X_v | Y) = 0$. This is exactly the statement $I_2(X_u, X_v | Y) = \mathrm{conn}_{G,\gamma}(u,v)$.

For the inductive step, assume the theorem when the BSC channels are given by $\delta' : E(G) \to [0,1]$ with $|S_{\delta'}| < |S_{\delta}|$. Pick an arbitrary edge $f \in S_{\delta}$. We will now interpolate between the case in which $\delta(f) = 0$, and the case in which $\delta(f) = 1$, with the other edge channels held fixed. For any $t \in [0,1]$, let $\delta_t : E(G) \to [0,1]$ be given by $\delta_t(e) = \delta(e)$ for $e \neq f$, and $\delta_t(f) = t$. Define corresponding spin synchronization models $(X_t, Y_t)$, and also $\gamma_t = \delta_t^2$. Write $$I(t) := I_2(X_{t,u}; X_{t,v} | Y_t), \mbox{ and } C(t) := \mathrm{conn}_{G,\gamma_t}(u,v).$$ In order to prove that $I(t) \leq C(t)$ for all $t \in [0,1]$, we need the following claim:
\begin{claim}\label{claim:termwisegrowthbound} There is non-decreasing $h : [0,1] \to \RR$ such that
\begin{equation*}\label{eq:ggrowthbound}I(t) = I(0) + (I(1) - I(0)) \cdot t^2 \cdot h(t),\end{equation*}
\end{claim}
Assume the claim is true. Then, $h(1) = 1$, and since $h(t)$ is non-decreasing, $h(t) \leq 1$. Hence, \begin{align*}I(t) &\leq I(0) + (I(1) - I(0)) \cdot t^2 \\ &= I(0) \cdot (1-t^2) + I(1) \cdot t^2 \\ &\leq C(0) \cdot (1-t^2) + C(1) \cdot t^2 = C(t).\end{align*} The inequality of the last line follows because $I(0) \leq C(0)$ and $I(1) \leq C(1)$ by the inductive hypothesis. The equality of the last line follows by the linearity of the connection probability in the parameter $\gamma_t(f) = t^2$.

It only remains to prove the claim. Write $E' = E(G) \sm f$. Also write $f = (i,j)$, $A_t = X_{t,u} \cdot X_{t,v}$, and $B_t = X_{t,i} \cdot X_{t,j}$. By Proposition \ref{prop:correqualschi2}, and because $Y_{t,E'}$ is a subset of $Y_t$, $$I(t) = \EE[\EE[A_t | Y_t]^2] = \EE[\EE[\EE[A_t | Y_t]^2 | Y_{t,E'}]].$$

Since the only edge channel to change with $t$ is $Q_f$, we can couple $X_0 = X_t$, and $Y_{0,E'} = Y_{t,E'}$.
Hence, it suffices to prove that the function $$h(t; Y_{0,E'}) := \frac{1}{t^2}\left(\EE[\EE[A_t | Y_t]^2 | Y_{t,E'}] - \EE[\EE[A_t | Y_0]^2 | Y_{0,E'}]\right)$$ is non-decreasing in $t$, since then we can set $h(t) = \sum_{\sigma \in \{-1,+1\}^{E'}} h(t; \sigma) \cdot \PP[Y_{0,E'} = \sigma],$ which will also be non-decreasing in $t$.

Fix $\sigma \in \{-1,+1\}^{E'}$ such that $\PP[Y_{t,E'} = \sigma] > 0$, and let $P_{\alpha\beta} = \PP[(A_t, B_t) = (\alpha,\beta) \mid Y_{t,E'} = \sigma].$ Set $a = P_{1,1}$, $b = P_{1,-1}$, $c = P_{-1,1}$, $d = P_{-1,-1}$. Since $Y_{t,f} \indy A_t,Y_{t,E'} | B_t$, one can explicitly calculate
\begin{align*}\EE[\EE[&A_t|Y_t]^2|Y_{t,E'}] \\ = &\bigg( \frac{((a(1-t) + b(1+t)) - (c(1-t)+d(1+t)))^2}{2((a(1-t) + b(1+t)) + (c(1-t)+d(1+t)))} \\ &+\frac{((a(1+t) + b(1-t)) - (c(1+t)+d(1-t)))^2}{2((a(1+t) + b(1-t)) + (c(1+t)+d(1-t)))}\bigg),\end{align*}
Plugging this in and simplifying, if $b = d = 0$ or $a = c = 0$, then $h(t; \sigma) = 0$, which is non-decreasing because it is constant. Otherwise we get $h(t; \sigma) = \frac{16(ad - bc)^2}{1-t^2(a-b+c-d)^2},$ which is non-decreasing on $[0,1]$ because $(a-b+c-d)^2 < (a+b+c+d)^2 = 1$.
This proves the claim.

\end{proof}
\begin{lemma}\label{lem:i2onpath}
Suppose $(X,Y)$ is a spin synchronization model on a path $P$ with endpoints $u$ and $v$. Then $$\mathrm{conn}_{P,\gamma}(u,v) = I_2(X_u; X_v | Y_{E(P)}).$$
\end{lemma}
\begin{proof}
 \begin{align*}&\mathrm{conn}_{P,\gamma}(u,v) = \textstyle \prod_{(i,j) \in E(P)} I_2(X_i; X_j | Y_{ij}) \\
 &= \textstyle \prod_{(i,j) \in E(P)} I_2(X_i \cdot X_j; Y_{ij}) \tag*{(Prop. \ref{prop:i2equalsprod})} \\
 &= I_2(\textstyle \prod_{(i,j) \in E(P)} X_i \cdot X_j; Y_{E(P)}) \tag*{(Prop. \ref{prop:chi2seriesmultiplicativity})} \\
 &= I_2(X_u \cdot X_v; Y_{E(P)}) = I_2(X_u; X_v | Y_{E(P)}) \tag*{(Prop. \ref{prop:i2equalsprod})}
 \end{align*}
\end{proof}

\begin{proof}[Proof (of Corollary \ref{cor:subadditivitybound})]

The corollary follows from Theorem \ref{thm:mainsymmetricbound}, the union bound $\mathrm{conn}_{G,\gamma}(u,v) \leq \sum_{P \in \cP_G(u,v)} \mathrm{conn}_{P,\gamma}(u,v),$ and Lemma \ref{lem:i2onpath}.
\end{proof}

\section*{Acknowledgements}
This research was partly supported by NSF CAREER Award CCF-1552131.

\printbibliography

\end{document}